\documentclass[reqno, 12pt]{amsart}
\usepackage[utf8]{inputenc}
\usepackage{amssymb,amsmath,amsfonts,amsthm,calrsfs,mathtools}
\usepackage{color}
\usepackage[english]{babel}
\usepackage[T1]{fontenc}
\usepackage{latexsym}
\usepackage{enumitem}
\usepackage{tabularx}
\usepackage[a4paper,top=3cm,bottom=2cm,left=3cm,right=3cm,marginparwidth=1.75cm]{geometry}
\usepackage[colorinlistoftodos]{todonotes}
\usepackage[colorlinks=true, allcolors=red]{hyperref}
\theoremstyle{plain}
\newtheorem{theorem}{Theorem}
\newtheorem{lemma}{Lemma}

\theoremstyle{definition}

\newtheorem{definition}{Definition}

\newtheorem{remark}{Remark}

\newtheorem{example}{Example}

\DeclareMathOperator*{\esssup}{ess\,sup}
\DeclareMathOperator*{\essinf}{ess\,inf}

\providecommand{\keywords}[1]

\title[Composition operators on functions space]{On the Composition Operator with variable integrability} 
\author{Carlos F. \'{A}lvarez}
\address{Instituto de Matemáticas, Pontificia Universidad Católica de Valparaíso, Blanco Viel 596, Cerro Barón, Valparaíso, Chile}
\email{carlos.alvarez.e@pucv.cl}

\author{Javier Henríquez-Amador}
\address{Departamento de Matemáticas, Universidad del Valle, Ciudadela Universitaria Meléndez
Edificio E20, Cali, Colombia}
\email{javier.henriquez@correounivalle.edu.co}

\date{\today}


\begin{document}

\begin{abstract}
In this note, we consider a class of composition operators on Lebesgue spaces with variable exponents over metric measure spaces. Taking advantage of the compatibility between the metric-measurable structure and the regularity properties of the variable exponent, we provide necessary and sufficient conditions for this class of operators to be bounded and compact, respectively. In addition, we show the usefulness of the variable change to study weak compactness properties in the framework of non-standard spaces.
\end{abstract}
\maketitle
\tableofcontents

\section{Introduction}
Let $(X,d,\mu)$ be a metric measure space equipped with a metric $d$ and the Borel regular
measure $\mu.$ If $\varphi: X\rightarrow X$ is a non-singular map, i.e., \begin{equation}\label{I1}
    \mu(\varphi^{-1}(E))=0, \ \ \textrm{for all Borel $\mu$-measurable set} \ E\subset X \ \ \textrm{with} \ \ \mu(E)=0,
\end{equation}
then we can define the composition operator $$C_{\varphi}: f \mapsto f\circ\varphi, \ \  \textrm{for every measurable function $f$ on} \ \ X.$$
According to (\ref{I1}) the measure $(\mu\circ\varphi^{-1})(\cdot):=\mu(\varphi^{-1}(\cdot))$ is absolutely continuous with respect to the measure $\mu$ therefore the Radon-Nikodym theorem implies that there exists a non-negative function $u_{\varphi}\in L^{1}(X)$ such that 
\begin{equation}\label{I1*}
    \mu(\varphi^{-1}(E))=\displaystyle\int_{E}u_{\varphi}(x) \ d\mu, \ \ \textrm{for all Borel }  \ \mu\textrm{-measurable set} \ E\subset X.  
\end{equation}
The composition operator \(C_{\varphi}\) appears naturally in the context of variable change and has recently been applied to dynamical systems, partial differential equations, and data science. For more details, see \cite{bevanda2021koopman,vcrnjaric2020koopman,hataya2024glocal,zhang2023quantitative}. For seminal applications, refer to \cite{koopman1931hamiltonian,koopman1932dynamical}.  
 
In $L^{p}$ spaces, characterizing of the boundedness of composition operators is a fundamental problem (e.g., see \cite{Bourdard2000,Rosenthal1995, SinnghManhas1993,Takagi1992}). For other function spaces, see \cite{AroraDattVerma2007,Bourduard2019, CowenMacCauler1995, gol2010homeomorphisms, ikeda2023boundedness, maligranda2011hidegoro, vodop1989mappings,vodopyanov2005composition}. It is well known that $C_{\varphi}$ continuously maps $L^{p}(X)$ into itself if and only if the function $u_{\varphi}$ is essentially bounded on $X$. In such a case, $\varphi$ is said to induce a composition operator on $L^p(X)$.

The first question that we address in this note is: 
\begin{itemize}
    \item[(Q1)]\label{Q1} What kind of control should be imposed on the variable exponent $p:X\rightarrow \mathbb{R}$ such that the map $\varphi$ induces a composition operator $C_{\varphi}$ on $L^{p(\cdot)}(X)$? 
\end{itemize}

A fruitful analysis of variable integrability spaces $L^{p(\cdot)}(X)$ is obtained when $X$ is a Euclidean domain, for example see \cite{cruz-uribe2013,diening2011lebesgue,diening2004open}. Unfortunately, the same cannot be said for general metric measure spaces. However, the authors in \cite{harjulehto2004variable} achieved boundedness of the maximal operator (see also \cite{kokilashvili2016integral} by endowing $X$) with a notion of “dimension” through a local uniformity condition:  
\begin{equation}\label{I2*}
\mu(B(x,r))\approx r^{p(x)}, \ \ x\in X,
\end{equation}
where, $B(x,r):=\{z\in X:d(z,x)<r\}, \ r\geq0.$

An essential difficulty in proving the boundedness of $C_{\varphi}$ over all $L^{p(\cdot)}(X)$ is that \textit{Cavalier's principle} (see \cite[Lemma~1.10]{bjorn2011nonlinear}) does not hold. This motivates our question about appropriate control conditions (Q1). A control on the variable exponent $p(\cdot)$ compatible with a condition like \eqref{I2*}, which we propose in this article, is: 
\begin{equation}\label{I3*}
     \inf_{x\in B}p(x)\leq \inf_{x\in \varphi^{-1}(B)}p(x)\leq\sup_{x\in \varphi^{-1}(B)}p(x)\leq  \sup_{x\in B}p(x).
 \end{equation}
 for every ball $B\subset X.$ Note that when $p(\cdot)$ is constant, the condition (\ref{I3*}) is trivial.  it seems natural to impose a “control” on the local supremum and infimum of the variable exponent as a new ingredient in our analysis. 
 
 Another important point is that our boundedness result for $C_{\varphi}$ over $L^{p(\cdot)}(X)$ can be extended in several directions including:   
\begin{itemize}
\item $n$-dimensional Euclidean domains on $\mathbb{R}^{n}$ with Lebesgue measure and Euclidean distance. 

\item Complete Riemannian manifolds of positive Riemannian measure and distance.  

\item Locally compact and separable group equipped with a left-invariant metric and left-invariant Haar measure.
\end{itemize}

Currently, the boundedness of $C_{\varphi}$ on variable integrability spaces such as $L^{p(\cdot)}(X)$ is not fully understood, even for $n$-dimensional Euclidean domains. However, boundedness results for these operators are known in more regular function spaces, such as variable exponent Bergman spaces (see \cite{morovatpoor2020boundedness}) and recently in holomorphic function spaces \cite{KR2022}. Additionally, the inequality
\begin{equation}\label{I4*}
\displaystyle\int |C_{\varphi}f(x)|^{p} \ d\mu \leq \ C \ \displaystyle\int |f(x)|^{p} \ d\mu, \ \ 1\leq p<+\infty,
\end{equation}
does not hold in $L^{p(\cdot)}(X)$ unless $p(\varphi(\cdot))=p(\cdot)$ almost everywhere in $X.$ 

Regarding compactness, it is well known that $L^{p}(X)$ with $p\geq1$ does not support compact composition operators if $X$ has not atoms. In this paper, we show that metric measure spaces satisfying a local uniform property as in $(\ref{I2*})$ have no atoms. This allows us to provide a different proof from the constant exponent case, using recent developments on precompact sets on $L^{p(\cdot)}(X)$ obtained in \cite{gorka2015almost}. 

We also study a related class of operators: 
$$T_{\varphi}:f\mapsto f\circ\varphi, \ \ D(T_{\varphi}):=\left\{f\in L^{p(\cdot)}(X): f\circ\varphi\in  L^{(p\circ\varphi)(\cdot)}(X)\right\}.$$
This class $T_{\varphi}$ represents the natural extension to the variable exponent case. Recent studies on boundedness, compactness, and closed-range properties for such operators have been conducted in \cite{bajaj2024composition,datt2024weighted} for bounded exponents defined on complete $\sigma$-finite spaces. 

In this paper, we study the boundedness of $T_{\varphi}$ by providing a simple proof on $L^{p(\cdot)}(X)$, where $X$ is a metric measure space with a doubling measure $\mu$ and unbounded exponents $p(\cdot)$. On the other hand, the non-compactness of $T_{\varphi}$ is obtained when $X$ is a connected space. Not least, in some cases, $T_{\varphi}$ maps $L^{p(\cdot)}(X)$ to $L^{p(\cdot)}(X)$. For example, when  $$T_{\varphi}: L^{p(\cdot)}(X) \rightarrow L^{(p \circ \varphi)(\cdot)}(X) \quad \text{and} \quad L^{(p \circ \varphi)(\cdot)}(X) \hookrightarrow L^{p(\cdot)}(X),$$ the right-hand embedding implies, in particular, that $p(\varphi(x)) \geq p(x)$ a.e. in $x \in X$. Moreover, this note implies the right inequality in (\ref{I3*}). In this sense, the results concerning the operator $C_{\varphi}$ are obtained with a weaker hypothesis by replacing the embedding with regularity of log-continuous type, which we adapt to the environment of metric measure spaces. 

Finally, another property that has not been explored in the framework of spaces with variable integrability is the weak compactness of the operator $(T_{\varphi}$, even in the Euclidean case. In this direction, we show that the operator $T_{\varphi}$ behaves well on weakly compact sets in $L^{p(\cdot)}([0,1])$, and some results are obtained in the non-reflexive setting $p^{-}=1$. In fact, our partial results allow us to state the following conjecture: Let $p^{-}=1$, $$T_{\varphi} \quad \text{is weakly compact on} \quad L^{p(\cdot)}([0,1]) \quad \text{if and only if} \quad \inf_{x\in [0,1]} u_{\varphi}(x)=0.$$

Let us now describe the organization of the article. In Section \ref{preli} we fix the notations and recall the definitions and a few results which will be important in our work. In Section \ref{continuity}, we study continuity and compactness for the operator $C_{\varphi}$. Finally, we study some properties of the operator $T_{\varphi}$ in the Section \ref{properties}.

\section{Preliminaries and notations}\label{preli}
We assume throughout the paper that $\mathcal{B}_{X}$ is the $\sigma$-algebra of Borel generated by $\mu$-measurable open sets in $X$ and the measure $\mu$ of  every open nonempty
set is positive and that the measure of every bounded set is finite on $X$.

\subsection{Doubling measure and \texorpdfstring{$Q$}{lg}-Ahlfors property} We define the property of doubling measure which endows a metric measure space with good properties, for more details see \cite{bjorn2011nonlinear, heinonen2015sobolev}. 

\begin{definition}A measure $\mu$ is said to satisfy the doubling condition if there exists a positive constant $C$ such that $$\mu(B(x,2r))\leq C \ \mu(B(x,r)), \ \ \textrm{for every ball} \ \ B(x,r).$$
\end{definition}
Another fruitful property for a metric measure space is the regular $Q$-Ahlfors property, which in some cases is stronger than the doubling property (see \cite{harjulehto2004variable}). 
\begin{definition} We say that the measure $\mu$ is lower Ahlfors $Q$-regular if there exists a positive 
constant $C$ such that $\mu(B) \leq C \ \textrm{diam}(B)^{Q}$ for every ball $B\subset X$ with $\textrm{diam} B\leq 2 \ \textrm{diam} X.$ We say that $\mu$ is upper Ahlfors $Q$–regular if there exists
a positive constant $C$ such that $\mu(B) \geq C \ \textrm{diam}(B)^{Q}$ for every ball $B \subset X$ with
$\textrm{diam} B \leq 2 \ \textrm{diam} X$. The measure $\mu$ is Ahlfors $Q$–regular if it is upper and
lower Ahlfors $Q$–regular; i.e., if $$\mu(B) \approx \textrm{diam}(B)^{Q} \  \ {\rm{ for \ every\  ball}}\  B \subset X \ {\rm{with}}\  {\textrm{diam}}\ B \leq 2 \ {\textrm{diam}}\ X.$$    
\end{definition}

\subsection{The variable exponents class}

The class of variable exponents, denoted by $\mathcal{P}(X)$, is defined by $$\mathcal{P}(X):=\left\{p:X\to [1,\infty): p(\cdot) \ \textrm{is Borel measurable}\right\}.$$ Given $A\subset X
$ and $p(\cdot)\in \mathcal{P}(X)$ we put $$p^{+}_{A}:= \displaystyle\esssup_{x\in A} p(x)\ \ {\rm{and}} \ \ p^{-}_{A}:= \displaystyle\essinf_{x\in A}p(x).$$ When the domain is clear we simply write $p^{+}= p^{+}(X)$ and $p^{-}=p^{-}(X).$ Some properties of regularity at infinity, relative to variable exponent, are useful for studying composition operators within the framework of non-standard functional spaces. For Euclidean domains and metric measures spaces, these properties are provided respectively in \cite{cruz-uribe2013,diening2011lebesgue} and \cite{gorka2015almost, harjulehto2004variable}. 

\begin{definition}
Let $p(\cdot)\in \mathcal{P}(X)$. We say that 
\begin{enumerate}
    \item $p(\cdot)$ is locally 
log-Hölder continuous, denoted by $p(\cdot)\in LH_{0}(X)$, if there exists a constant $K_0$ such that for all $x,y\in X, \  d(x,y)<\frac{1}{2}$, $$ |p(x)-p(y)|\leq \frac{K_0}{-\log(d(x,y))}. $$

\item $p(\cdot)$ is log-Hölder continuous at infinity with point base $x_0 \in X$, denoted by $p(\cdot) \in LH_{\infty}(X)$, if there exist constants $K_{\infty}$ and $p_{\infty}$ such that for all $x\in X$,
$$|p(x)-p_{\infty}|\leq \frac{K_{\infty}}{\log(e+d(x,x_0))}.$$
\end{enumerate}
When $p(\cdot)$ is log-Hölder continuous both locally and at infinity, we denote this by $p(\cdot)\in LH(X).$  
\end{definition}
For this work, we introduce a class of exponents associated with the $\varphi$-map. Denote by $\mathcal{B}_0$ the set of
all open ball in $X$.

\begin{definition}
Let $\varphi: X\to X$ be a Borel measurable map, we define the following kinds of exponents:
$$\mathcal{P}_{\varphi^{+}}^{\textrm{log}}(X):=\{p(\cdot)\in LH(X): [\varphi]_{p^{+}}\geq 1\},$$ $$\mathcal{P}_{\varphi^{-}}^{\textrm{log}}(X):=\{p(\cdot)\in LH_0(X): [\varphi]_{p^{-}}\leq 1\},$$ $$\mathcal{P}_{\varphi}^{\textrm{log}}(X):=\mathcal{P}_{\varphi^{+}}^{\textrm{log}}(X)\cap \mathcal{P}_{\varphi^{-}}^{\textrm{log}}(X),$$ where
$$[\varphi]_{p^{+}}:=\displaystyle\inf\left\{\frac{p^{+}_{B}}{p_{\varphi^{-1}(B)}^{+}}:   B\in\mathcal{   B}_{0}\right\}, \ \ \ [\varphi]_{p^{-}}:=\displaystyle\sup\left\{\frac{p^{-}_{B}}{p_{\varphi^{-1}(B)}^{-}}:B\in\mathcal{B}_{0}\right\}.$$
\end{definition}
In the following example, we show that $[\varphi]_{p^+}\geq 1$  for countable sub-covers, in fact, this is sufficient for the continuity result that we will obtain below (see Theorem \ref{t2*}).
\begin{example}\label{ex1}
    Let $X:=\mathbb{R}^{+}$ and define $p:\mathbb{R}^{+}\to[1,+\infty)$ a variable exponent given by 
    \begin{equation*}
     p(x):=\left\{\begin{array}{cc}
       1+x^{2},   & x\in(0,1) \\
        
        1+\frac{1}{2x-1},  & x\geq 1 
     \end{array} \right.
        \end{equation*}
    It is not difficult to show that $$0\leq p(x)-1\leq \displaystyle\frac{\ln(1+e)}{\ln(|x|+e)}, \ \ \textrm{for all} \ \ x\in\mathbb{R}^{+}.$$
    That is, $p(\cdot)\in LH_{\infty}(\mathbb{R}^{+}).$ In addition, it is easy see that $p(\cdot)\in LH_{0}(\mathbb{R}^{+});$ for $x\in(0,1)$ and $y\geq1$ such that it follows that $$|p(x)-p(y)|=\displaystyle\left|\frac{2x^{2}y-x^{2}-1}{2y-1}\right|\leq 2x^{2}y-2x^{2}\leq 2 \ (y-1)\leq 2 \ (y-x)=2 \ |x-y|.$$
    Similarly, for $x\geq 1$ and $y\in(0,1).$ This is sufficient to show that $p(\cdot)$ is Lipschitz continuous on $\mathbb{R}^{+}$and therefore $p(\cdot)\in LH_0(\mathbb{R}^{+}).$ Now, 
    consider $\varphi:\mathbb{R}^{+}\to \mathbb{R}^{+}$ be a non-singular map such that  $0\leq\varphi(x)\leq x$  for all $x\in\mathbb{R}^{+}.$ On the other hand, let
    $\epsilon\in(0,+\infty)$ and we consider $$\mathcal{A}:=\{(x-\epsilon,x+\epsilon):x\in\mathbb{R}^{+}\}.$$
    Let $\{q_j\}_{j}\subset\mathbb{Q}^{+}$ be such that $$q_j\to+\infty  \ \ \textrm{as} \ \ j\to+\infty, \ \ \textrm{and} \ \ \mathbb{R}^{+}=\bigcup_{j}I_j, \ \ I_{j}:=(q_j-5\epsilon,q_j+5\epsilon).$$ Then,  by using $LH_{\infty}$-regularity we have 
    \begin{eqnarray*}
        \sup_{x\in\varphi^{-1}(I_j)}p(x)&\leq& \sup_{x\in\varphi^{-1}(I_j)} |p(x)-p_{\varphi}(x)|+\sup_{x\in\varphi^{-1}(I_j)} p_{\varphi}(x) \\
        &\leq& \sup_{x\in\varphi^{-1}(I_j)}\frac{C_{\infty}}{\ln(e+\varphi(x))}+p_{I_j}^{+} \\
        &\leq& \frac{C_{\infty}}{\ln(e+q_{j}-\epsilon)}+ p_{I_j}^{+}. 
    \end{eqnarray*}
    Hence, since $q_j\rightarrow+\infty$ as $j\to +\infty$ $$\limsup_{j} p_{\varphi}^{+}(I_j)\leq \limsup_{j} p^{+}(I_j)<+\infty, \ \ \textrm{as} \ \ j\to+\infty.$$
    Therefore,  $$0\leq  p^{+}(I_{q_{n_j}})- p_{\varphi}^{+}(I_{q_{n_{j}}}), \ \ \textrm{for many infinite} \ \ n_{j}\in \mathbb{N} .$$
    Finally, in particular note that $L^{p(\cdot)}(\mathbb{R}^{+})$ is not embedded in $L^{p_{\varphi}(\cdot)}(\mathbb{R}^{+})$ and $p(\cdot)\in \mathcal{P}_{\varphi^+}^{\log}(\mathbb{R}^{+})$ (in the sense of accounting coverages).
\end{example}



\subsection{Variable integrability spaces} 
The space $L^{p(\cdot)}(X)$ is the classical variable Lebesgue space on $X$, some of its basic properties have been studied, for instance in \cite{harjulehto2004variable}. 

\begin{definition} Let $p(\cdot)\in \mathcal{P}(X)$ and $L_{0}(X)$ the set of measurable functions on $X.$
 The \textit{Lebesgue space with variable exponent} $L^{p(\cdot)}(X)$ is defined by $$L^{p(\cdot)}(X):=\left\{ f\in L_{0}(X):\rho\left(\frac{f}{\eta}\right)<\infty \ {\rm{for \ some \ \eta >0}}\right\},$$
\end{definition}
where $$\rho(f):=\displaystyle\int_{X}|f(x)|^{p(x)} d\mu(x),$$ is called the \textit{associated modular} with $p(\cdot)$. It is known that $L^{p(\cdot)}(X)$ is a Banach space equipped with the \textit{Luxemburg-Nakano norm} $$\left\|f\right\|_{p(\cdot)}:=\inf\left\{\eta>0: \rho\left(\frac{f}{\eta}\right)\leq 1\right\}.$$
Additionally, when $p(\cdot)$ is bounded we obtain the inequality that provides a relation between the norm and the modular:
\begin{equation}\label{eq2}
\min\left\{\left\|f\right\|_{p(\cdot)}^{p^{+}}, \left\|f\right\|_{p(\cdot)}^{p^{-}}\right\} \leq \rho\left(f\right)\leq \max \left\{\left\|f\right\|_{p(\cdot)}^{p^{+}}, \left\|f\right\|_{p(\cdot)}^{p^{-}}\right\},
\end{equation}
for $f\in L^{p(\cdot)}(X).$ Also, in the case $p^{+}$ is finite the space $L^{p(\cdot)}(X)$ can be defined as all measurable functions $f:X\rightarrow\mathbb{R}$ such that $\rho(f)<+\infty$ and the dual space is identified with $L^{p'(\cdot)}(X)$ where $p'(\cdot)$ is the conjugate exponent relative to $p(\cdot)$, that is, $$\displaystyle\frac{1}{p'(x)}+\frac{1}{p(x)}=1, \ \ x\in X\setminus\{x:p(x)=1\}.$$
It is well-know that the space of bounded essentially functions with compact support $L^{p(\cdot)}_{c}(X)$ is dense in $L^{p(\cdot)}(X).$

\section{Some properties for   $C_{\varphi}$}\label{continuity}
In this section, we start analyzing the behavior and properties of the operator $C_{\varphi}$ on $L^{p(\cdot)}(X)$ based on the following results in the space $L^{p}(X):$
\begin{enumerate}
    \item[$(P_1)$] $C_{\varphi}:L^{p}(X)\rightarrow L^{p}(X)$ is bounded if and only if $u_{\varphi}\in L^{\infty}(X).$ In particular,
    $$\displaystyle\int_{X}|(C_{\varphi}f)(x)|^{p} \ d\mu(x) \ \leq \|u_{\varphi}\|_{\infty} \ \int_{X} |f(x)|^{p} \ d\mu(x), \ \textrm{for all} \ f\in L^{p}(X).$$
\end{enumerate}
 In the second part we completely characterize the continuity and compactness of operator $C_{\varphi}$ on $L^{p(\cdot)}(X).$  
\subsection{On the property $(P_1)$ in $L^{p(\cdot)}(X)$} 
Our first result shows that the integral inequality does not hold in general for any variable exponent, unless the exponent is both bounded and invariant under dilation's or contractions induced by $\varphi.$ The variable exponent induced by $\varphi$ is denoted by $p_{\varphi}(\cdot):=p(\varphi(\cdot)).$ Additionally we associate  to $\varphi$ the number $$\mathcal{U}(\varphi):=\sup_{B\in\mathcal{B}_{0}}\displaystyle\frac{\mu(\varphi^{-1}(B))}{\mu(B)}, \ {\rm{where}} \ \mathcal{B}_{0}:=\left\{B\in\mathcal{B}_{X}: B \ \textrm{is a ball in} \ X\right\}.$$  
  
\begin{theorem}\label{t1} Assume $p(\cdot)\in\mathcal{P}(X)$ with $p^{+}<+\infty$  and let $\varphi:X\rightarrow X$ a non-singular measurable transformation. Then, the following statements are equivalent:
  \begin{enumerate}
      \item[(M1)] There exists a constant $C>0$ such that 
      \begin{equation}\label{m1}
      \displaystyle\int_{X}|(C_{\varphi}f)(x)|^{p(x)} \ d\mu \ \leq C \ \int_{X} |f(x)|^{p(x)} \ d\mu, \ \textrm{for all} \ f\in L^{p(\cdot)}(X).
      \end{equation}
      \item[(M2)]  The function $u_{\varphi}:X\rightarrow\mathbb{R}$ is essentially bounded on $X$, and $$p(x)=p_{\varphi}(x), \  \ \textrm{a.e. in x} \in X.$$ 
    \end{enumerate}
 \end{theorem}
\begin{proof} It is clear that $(M.2) \Rightarrow(M.1)$ with $C:=\|u_{\varphi}\|_{\infty}.$ To prove $(M.1) \Rightarrow(M.2),$ first suppose that $p^{+}<\infty;$  since $\varphi$ is non-singular we have $p\circ\varphi\in L^{\infty}(X)$ so $$\displaystyle\frac{p_{\varphi}(\cdot)}{p(\cdot)}\in L^{\infty}(X).$$ 
  Assume that $(M.1)$ holds, so for $B\in \mathcal{B}_{0},$ with $\mu(B)<+\infty.$  Considering the function $f= \mathcal{X}_{B} \in L^{p(\cdot)}(X),$ from inequality (\ref{m1}) we deduce that
  \begin{equation}\label{m2}
   \frac{C \ \mu(B)}{\mu[\varphi^{-1}(B)]}\geq 1. 
   \end{equation}
  Now, let $\Omega_{\varphi}:=\{x\in X: p(x)\neq p_{\varphi}(x)\}$ and suppose that $\mu(\Omega_{\varphi})>0.$ Thus, if $E_{\varphi}:=\{x\in X:p(x)>p_{\varphi}(x)\}\in\mathcal{B}$,  then by $\sigma$-additivity we have
  $$\mu(E_{\varphi})>0 \ \ \textrm{or} \ \ \mu(\Omega_{\varphi}\setminus E_{\varphi})>0.$$ So, we see only the case $\mu(E_{\varphi})>0$ because the case $\mu(\Omega_{\varphi} \setminus E_{\varphi})>0$ is analogue with minor settings. In fact, since $\mu(E_{\varphi})>0$ then by \cite[Lemma~3.3.31]{heinonen2015sobolev} we can take $y\in E_{\varphi}$ such that  satisfies $$\mu(E_{\varphi}\cap B_{y})>0, \ \textrm{ for every ball} \ \ B_{y}\in \mathcal{B}_{0}.$$ Hence, since $\varphi^{-1}(X)=X$ from (\ref{m2}) we can choose $B_{y}\in\mathcal{B}_{0}$ (sufficiently rate large) such that 
  \begin{equation}\label{m3}
  0<\mu(E_{\varphi}\cap \varphi^{-1}(B_{y}))<+\infty.
  \end{equation}
  Hence, using (\ref{m1}) with the functions $f_{n}(\cdot)=n^{\frac{1}{p(\cdot)}} \mathcal{X}_{B_{y}}(\cdot)\in L^{p(\cdot)}(X),$ we obtain that  
  \begin{eqnarray*}
\displaystyle\int_{E_{\varphi}\cap\varphi^{-1}(B_{y})} |n|^{\frac{p(x)}{p_{\varphi}(x)}} \ d\mu&=& 
\int_{X}|(f_{n}\circ\varphi)(x)|^{p(x)} \ d\mu \\
&\leq& C \ \int_{X}|f_{n}(x)|^{p(x)} \ d\mu \\
&=& C \  n \ \mu(B_{y}).  
\end{eqnarray*}
Thus, $$\displaystyle\int_{E_{\varphi}\cap \varphi^{-1}(B_{y})} |n|^{\frac{p(x)}{p_{\varphi}(x)}-1} \ d\mu \leq C \ \mu(B_{y}).$$
By the \textit{classical Jensen Inequality} we have $$ \displaystyle\int_{E_{\varphi}\cap\varphi^{-1}(B_{y})} \left(\frac{p(x)}{p_{\varphi}(x)}-1\right) \ d\mu \leq \log^{-1}(n) \ \mu(E_{\varphi}\cap\varphi^{-1}(B_{y}))\log\left(\frac{C \ \mu(B_{y})} {\mu(E_{\varphi}\cap\varphi^{-1}(B_{y}))}\right)$$
so, taking $n\rightarrow+\infty$ $$\displaystyle\int_{E_{\varphi}\cap\varphi^{-1}(B_{y})} \left(\frac{p(x)}{p_{\varphi}(x)}-1\right) \ d\mu=0.$$ Therefore, $\mu(E_{\varphi}\cap\varphi^{-1}(B_{y}))=0$ which is in contradiction with (\ref{m3}). Consequently, $\mu(E_{\varphi})=0.$ Besides, $u_{\varphi}\in L^{+\infty}(X)$ follows from (\ref{m2}) with $B=E\in \mathcal{B}_{X}.$ 
\end{proof}

\subsection{Continuity for \texorpdfstring{$C_{\varphi}$}{lg}}
In this section, we extend the classical continuity result in standard $L^{p}$ spaces to function spaces with variable integrability $L^{p(\cdot)}(X)$. The proof strategy is inspired by the argument given by Cruz-Uribe and Fiorenza on the boundedness of the maximal operator in $L^{p(\cdot)}(X)$ when $X$ is a Euclidean domain (see \cite[Theorem~3.16]{cruz-uribe2013}). For this purpose, we adapt \cite[Lemma~3.26, Lemma~3.24]{cruz-uribe2013} to the setting of metric measure spaces with Ahlfors regularity; see also \cite{harjulehto2004variable}.
 
\begin{lemma}\label{l1}
Let $\varphi:X\rightarrow X$ be non-singular Borel map. Then, 
\begin{itemize}
   \item [(i)] If $p(\cdot)\in LH_{0}(X)$ and $\mu$ is lower Ahlfors $Q$-regular then there exists a positive constant $C$ such that 
    $$\mu(B)^{p(\varphi(x))-p^{+}_{B}}\leq C, \ \ \textrm{for all} \ \  B\in\mathcal{B}_{0} \ \textrm{and} \  x\in\varphi^{-1}(B).$$
    \item [(ii)] If $p(\cdot)\in LH_{\infty}(X)$ and $\mu$ is upper Ahlfors $Q$-regular then there exist positive constants $C_1, C_2$ such that for all function $f$ with $0\leq f\leq 1$ on $E\in \mathcal{B}$ we have $$\displaystyle\int_{\varphi^{-1}(E)}|f(\varphi(x))|^{p_{\infty}} \ dx\leq C_{1}\int_{\varphi^{-1}(E)}|f(\varphi(x))|^{p(x)} \ dx +C_2.$$
    Similarly, $$\displaystyle\int_{\varphi^{-1}(E)}|f(\varphi(x))|^{p(x)} \ dx\leq C_{1}\int_{\varphi^{-1}(E)}|f(\varphi(x))|^{p_{\infty}} \ dx +C_2.$$
    \end{itemize}
\end{lemma}

\begin{proof} 

$(i)$ Since $p(\varphi(x))-p^{+}_{B}\leq0$ for $x\in \varphi^{-1}(B)$, it suffices to check $(i)$ for balls $B:=B(v,r)$ with $r\leq 1/2.$   For $y_{0}\in B$ such that $$p_{B}^{+}\leq p(y_0)+\frac{1}{\log(1/2r)},$$  from $p(\cdot)\in LH(X)$ it follows that $$p^{+}_{B}-p(\varphi(x))\leq |p(\varphi(x))-p(y_{0})|+1\leq \frac{C_{0}}{-\log d(\varphi(x),y_{0})}\leq \frac{C_{0}}{\log(1/2r)}.$$
So, $$\log (2r)^{p(\varphi(x))-p_{B}^{+}}\leq C_{0}, \ \ x\in\varphi^{-1}(B).$$
Therefore, since $\mu$ is lower Ahlfors $Q$-regular and $p(\varphi(x))-p_{B}^{+}\leq 0$ we have  $$\mu(B)^{p(\varphi(x))-p_{B}^{+}}\leq (Cr^{Q})^{p(\varphi(x))-p_{B}^{+}} 
    \leq C \ r^{Q[{p(\varphi(x))-p_{B}^{+}]}}\leq C_{Q}.$$
    $(ii)$ For $x_{0}\in X,$ define the function $h_{x_0}: X\rightarrow\mathbb{R}$ given by $h_{x_0}(x):=(e+d(x,x_{0}))^{-r}, \ \  x\in X.$ First, we show that $$h_{x_{0}}^{p^{-}}\in L^{1}(X) \ \ \textrm{provided that} \ \ rp^{-}\in(q,+\infty).$$ Indeed, for the base point $x_0\in X$ we consider the countable collection $\{C_{j}: j\in\mathbb{N}\}$ where, $$C_{j}:=B(x_0,2^{j})\setminus B(x_0,2^{j}-1)\in \mathcal{B}_{0}, \ \ 
    \textrm{for each} \ \ j\in\mathbb{N}.$$    Hence, 
    \begin{eqnarray*}
        \displaystyle\int_{X} |h_{x_0}(x)|^{p^{-}} \ d\mu &=& \sum_{j=1}^{+\infty} \int_{C_j} |h_{x_0}(x)|^{p^{-}} \ d\mu \\
        &=& \sum_{j=1}^{+\infty} \int_{B_{2^{j}}\setminus B_{2^{j}-1}} \left(\frac{1}{e+d(x,x_0)}\right)^{rp^{-}} \ d\mu\\
        &\leq& \sum_{j=1}^{+\infty} \int_{B_{2^{j}}\setminus B_{2^{j}-1}} 2^{-jrp^{-}} d\mu\\
        &\leq& \sum_{j=1}^{+\infty} \mu(B_{2^{j}}) \  2^{-jrp^{-}} \\
        &\leq& 4^{q}\sum_{j=1}^{+\infty}    2^{(q-rp^{-})j}<+\infty \ \ \ (rp^{-}\in (q,+\infty)). \end{eqnarray*}
    On the other hand, decompose $\varphi^{-1}(E):=F_{1}\cup F_{2}$ where $$F_{1}:=\{x:f(\varphi(x))\leq h_{x_{0}}(x)\}, \ \ \ \ F_{2}:=\{x:f(\varphi(x))>h_{x_0}(x)\}.$$
    On $F_{1} $ we have, $$\displaystyle\int_{F_1} f(\varphi(x))^{p_{\infty}} \ 
 d\mu\leq \int_{F_1} f(\varphi(x))^{p^{-}}  \ 
 d\mu \leq \int_{F_1} h_{x_0}(x)^{p^{-}} \ d\mu.$$
  By the $LH_{\infty}$-regularity, 
  $$h_{x_{0}}(x)^{-|p(x)-p_{\infty}|}\leq \exp{(r\log(e+d(x,x_0))|p(x)-p_{\infty}|)}\leq \exp{(rC_{\infty})}. $$
    Finally, since $f(\varphi(x))\leq 1$ we get,
    $$\displaystyle\int_{F_1} f(\varphi(x))^{p_{\infty}} \ d\mu \leq \int_{F_1} f(\varphi(x))^{p(x)} \ h_{x_0}(x)^{-|p(x)-p_{\infty}|} \ d\mu$$
    $$\leq \exp{(rC_{\infty})}\int_{F_1} f(\varphi(x))^{p(x)} \ d\mu.$$
\end{proof}

With the previous lemma, the main result of this work in response to Question \ref{Q1} is the following.

\begin{theorem}\label{t2*}
Assume that $\mu$ is a $Q$-Ahlfors regular measure on $X$, let $\varphi: X \rightarrow X$ be a non-singular map. Then:
\begin{enumerate}
    \item[$(c1)$] Let $p(\cdot) \in \mathcal{P}_{\varphi^{+}}^{\textrm{log}}(X).$ If the map $u_{\varphi}$ is essentially bounded on $X$, then the operator $C_{\varphi}$ maps $L^{p(\cdot)}(X)$ into itself.
    \item[$(c2)$] Let $p(\cdot) \in \mathcal{P}_{\varphi^-}^{\textrm{log}}(X).$ If the operator $C_{\varphi}$ maps $L^{p(\cdot)}(X)$ into itself, then there exists a positive constant $C$ such that $\mu(\varphi^{-1}(B)) \leq C \, \mu(B)$ for every ball $B \in \mathcal{B}_{0}$, i.e.,  
    \[
    \mathcal{U}(\varphi) < +\infty.
    \]
    \item[$(c3)$] Assume $p(\cdot) \in \mathcal{P}_{\varphi}^{\textrm{log}}(X).$ If, in addition, $\mu$ is a doubling measure on $X$, then $(c1)$ and $(c2)$ are equivalent.
\end{enumerate}
\end{theorem}
\begin{proof} $(c1).$ Suppose that $u_{\varphi}\in L^{\infty}(X).$ Since $p^{+}<+\infty$, by density it is sufficient to show that  $$C_{\varphi}:L^{\infty}_{c}(X)\cap L^{p(\cdot)}(X)\rightarrow L^{p(\cdot)}(X).$$ Indeed, let $f\in L_{c}^{\infty}(X)\cap L^{p(\cdot)}(X)$ with $\|f\|_{p(\cdot)}\leq 1$ and consider the decomposition $|f|=f_1+f_2$ where $f_1 :=|f|\chi_{\{|f|>1\}}$ and $f_2:= |f|\chi_{\{|f|\leq 1\}}.$ We will divide the test into two steps: 

\vspace{0.2cm}
 \textbf{Step 1.} Let us see that there exists a constant $C>0$ independent of $f_{1}$ such that 
\begin{equation}\label{10}
\displaystyle\int_{X} |(f_{1}\circ\varphi)(x)|^{p(x)} \ dx \leq C.
\end{equation}
Since $f_{1}\in L^{\infty}(\mathbb{R})$ and $X$ is separable it follows that the open cover of $X$$$\left\{B(x,5r_{Q}):x\in X\right\},$$   with  $K:=\esssup_{x\in X}(f_{1}\circ\varphi)(x)>1$ and $r_{Q}:=K^{-\frac{1}{Q}},$ admits a pairwise disjoint countable open sub-cover $\{B^{j}_{5r_{Q}}: j\in\mathbb{N}\}$ such that $$X\subset \bigcup_{j\in\mathbb{N}} B^{j}_{r_{Q}}.$$ From the $\mu$-regularity, $\mu(B^{j}_{r_{Q}})\leq C_{Q} \ K^{-1}$ for all $j\in\mathbb{N}.$ In addition, since $
p(\cdot)\in LH_{0}(X)$ by using Lemma \ref{l1} there exists $M>0$ such that for $B\in \mathcal{B}_{0}$ and $x\in\varphi^{-1}(B)$ we have $\mu(B)^{p_{\varphi}(x)-p^{+}(B)}\leq M$ and the fact that $K\in(1,+\infty)$ taking $r(\cdot):=p(\cdot)-p_{\varphi}(\cdot)$ for $x\in X$ we get 
\begin{eqnarray*}
   |(f_{1}\circ\varphi)(x)|^{p(x)} &=& |(f_{1}\circ\varphi)(x)|^{r(x)} \ |(f_{1}\circ\varphi)(x)|^{p_{\varphi}(x)}\\
   &\leq&  \left(|(f_{1}\circ\varphi)(x)|^{r(x)} \ \mathcal{X}_{\{r(x)\geq0\}}+1\right) \ |(f_{1}\circ\varphi)(x)|^{p_{\varphi}(x)}\\
   &\leq& \left(K^{r(x)}\mathcal{X}_{\{r(x)\geq0\}}+1 \right) \ |(f_{1}\circ\varphi)(x)|^{p_{\varphi}(x)} \\
   &=&\left(\mu(B^{j}_{r_{Q}})^{-r(x)}\mathcal{X}_{\{r(x)\geq0\}}+1\right) \ |f_{1}\circ\varphi)(x)|^{p_{\varphi}(x)}.
\end{eqnarray*}
So, $$|(f_{1}\circ\varphi)(x)|^{p(x)}\leq \left(\mu(B^{j}_{r_{Q}})^{-r(x)}\mathcal{X}_{\{r(x)\geq0\}}+1 \right) \ |(f_{1}\circ\varphi)(x)|^{p_{\varphi}(x)}, \ \textrm{for} \ \  x\in X.$$
From $[\varphi]_{p^{+}}\leq 1,$ $$p^{+}_{\varphi}(B^{j}_{r_{Q}})\leq p^{+}(B^{j}_{r_{Q}}), \ \ \textrm{for all} \ \ j\in\mathbb{N}.$$
Hence, 
\begin{eqnarray*}
    \displaystyle\int_{X} |(f_{1}\circ\varphi)(x)|^{p(x)} \ dx &=&\sum_{j}\int_{\varphi^{-1}(B^{j}_{r_Q})} |(f_{1}\circ\varphi)(x)|^{p(x)} \ dx\\
    &\leq&\sum_{j} \int_{\varphi^{-1}(B^{j}_{r_Q})}(K^{p(x)-p_{\varphi}(x)}+1) |(f_{1}\circ\varphi)(x)|^{p_{\varphi}(x)} \ dx \\
    &\leq& \sum_{j} \int_{\varphi^{-1}(B^{j}_{r_Q})}(\mu(B^{j}_{r_Q})^{p_{\varphi}(x)-p^{+}_{B^{j}}}+1) |(f_{1}\circ\varphi)(x)|^{p_{\varphi}(x)} \ dx  \\
    &\leq& (M+1) \sum_{j} \int_{\varphi^{-1}(B^{j}_{r_Q})} |(f_{1}\circ\varphi)(x)|^{p_{\varphi}(x)} \ dx\\
    &\leq&(M+1)\|u_{\varphi}\|_{\infty}.  
\end{eqnarray*}
Therefore, the inequality (\ref{10}) is obtained with $C_{1}:=(M+1)\mathcal{U}(\varphi).$ Now, we estimate the size of $C_{\varphi}f_{2}.$

\textbf{Step 2.} There exists a constant $C_2>0$ independent of $f_{2}$ such that 
\begin{equation}\label{11}
    \displaystyle\int_{X} |(f_{2}\circ\varphi)(x)|^{p(x)} \ dx \leq C_{2}+\int_{X} h_{x_0}(x)^{p_{-}} \ dx.
\end{equation}
Since $p(\cdot)\in LH_{\infty}(X),$ we get $f_{2}\in L^{p_{\infty}}(X)$ and $u_{\varphi}\in L^{\infty}(X)$ implies also that $C_{\varphi}f_{2}\in L^{p_{\infty}}(X).$ Hence, for some $C>0$ independent of $f_{2}$ we obtain
\begin{eqnarray*}
\displaystyle\int_{X} |(f_{2}\circ\varphi)(x)|^{p(x)} \ dx &\leq & C \ \int_{X}   |(f_{2}\circ\varphi)(x)|^{p_{\infty}} \ dx + \int_{X} h_{x_0}(x)^{p_{-}} \ dx \\
&\leq & C \ \|u_{\varphi}\|_{\infty}\int_{X}   |f_{2}(x)|^{p_{\infty}} \ dx + \ \int_{X} h_{x_0}(x)^{p_{-}} \ dx \\ 
&\leq& C_{1}+\int_{X} h_{x_0}(x)^{p_{-}} \ dx, 
\end{eqnarray*}
 this last integral is finite by Lemma \ref{l1}.  Therefore, from (\ref{10}) and (\ref{11}) we obtain $$\|f\circ\varphi\|_{p(\cdot)}\leq C_{\varphi,p(\cdot)} \ \|f\|_{p(\cdot)}, \ \ \textrm{for all} \ \ f\in L_{c}^{\infty}(X)\cap L^{p(\cdot)}(X).$$
$(c2).$ Suppose that $$C_{\varphi}:L^{p(\cdot)}(X)\rightarrow L^{p(\cdot)}(X).$$ By the inequality (\ref{eq2}) and the relation given by (\ref{I1*}) it is sufficient to consider the case when $B\in\mathcal{B}_{0}$ is such that $$\|\mathcal{X}_{\varphi^{-1}(B)}\|_{p(\cdot)},\|\mathcal{X}_{B}\|_{p(\cdot)}\leq 1.$$ In this case, since $p(\cdot)\in LH_{0}(X)$ and $[\varphi]_{p^{-}}\leq 1$ applying Lemma \ref{l1} we can find a constant $C
>0$ such that,
$$\mu(\varphi^{-1}(B))\leq \|\mathcal{X}_{\varphi^{-1}(B)}\|_{p(\cdot)}^{p_{\varphi}^{-}}\leq C \  \|\mathcal{X}_{B}\|_{p(\cdot)}^{p_{\varphi}^{-}}\leq C \ \mu(B)^{\frac{p_{\varphi}^{-}}{p^{+}}}$$
$$\leq C \ \mu(B)^{\frac{1}{p^{+}_{B}}(p_{B}^{-}-p^{+}_{B})} \mu(B)= C \ \mu(B)^{\frac{1}{p^{+}}(p_{B}^{-}-p(x))} \ \mu(B)^{\frac{1}{p^{+}}(p(x)-p_{B}^{+})} \ \mu(B)$$ $$\leq C' \ \mu(B).$$ 
Finally, in order to obtain $(c3)$ by assuming doubling property of $\mu$, note that $$\displaystyle\frac{\mu(\varphi^{-1}(B))}{\mu(B)}=\frac{1}{\mu(B)}\int_{B}u_{\varphi}(x) \ d\mu(x), \ \ \textrm{for every ball} \ \ B.$$
    Hence, by differentiation it is easy see that $\mathcal{U}(\varphi)$ is finite if and only if $u_{\varphi}$ is in $L^{\infty}(X).$
\end{proof}

\begin{remark}\label{r1*} According to the proof of the above theorem: 

\begin{itemize}
    \item Another hypothesis to obtain $c1$ is $p_{\varphi}(\cdot)\geq p(\cdot)$ a.e. in $X$ and $p(\cdot)\in LH_{\infty}(X)$. Note that in our result we replace $p_{\varphi}(\cdot)\geq p(\cdot)$ by the hypothesis $p(\cdot)\in LH_{0}(X)$ and $[\varphi]_{p^{+}}\leq 1$ which also replaces the embedding $L^{p_{\varphi}(\cdot)}(X)\hookrightarrow L^{p(\cdot)}(X)$ proposed in \cite[Theorem~3.4]{bajaj2024composition}. 
\end{itemize}

\begin{itemize}
    \item In the proof of the above theorem, note that the condition (\ref{I3*}) applies to a uniform Vitali cover of the space. In this sense, control over the inductor map $\varphi$ and the exponent $p(\cdot)$ can be relaxed as shown in Example \ref{ex1}. In fact, note that for a suitable $\varphi$ (e.g. $\varphi(x):=Ax$ , $x\in\mathbb{R}^{+}$) it suffices to note that $$\displaystyle\int_{\mathbb{R}^{+}} (f_1\circ\varphi)(x)^{p(x)} \ dx \approx \int_{\varphi^{-1}({I_{q_{n_J}}})} (f_1\circ\varphi)(x)^{p(x)} \ dx,  \ \ \textrm{as} \ \ j\to +\infty. $$
    
\end{itemize}
\end{remark}    

\subsection{Compactness 
 for $C_{\varphi}$}\label{compactness}
We start this section by showing that $L^{p(\cdot)}(X)$ does not support non-trivial compact composition operators $C_{\varphi}$. Besides, we approach recent results relative to weak compactness in variable Lebesgue space and we provide some properties for $C_{\varphi}.$ 

\begin{lemma}\label{l2} 
Let $\mu$ be a $Q$-Ahlfors regular measure on $X$, $\varphi: X \rightarrow X$ be a non-singular Borel map, and $p(\cdot) \in LH_{0}(X)$ such that $[\varphi]_{p^{-}} = 1$. There exists a positive constant $C_{0}$ such that, given $A \in \mathcal{B}_{0}$ with $\mu(A) > 0$, if for any ball $B$ with $A \cap B \neq \emptyset$ and $\mu(B)<1$ then
$$\mu(B)^{1 - \frac{p(x)}{p_{\varphi}(x)}} \geq C_{0}, \quad \text{for all} \quad x \in \varphi^{-1}(A \cap B).$$
\end{lemma}

\begin{proof}
     On one hand, $$\sup_{x\in \varphi^{-1}(A\cap B)}p_{\varphi}(x)\leq \sup_{x\in A\cap B}p(x)\leq \sup_{x\in B} p(x)=p^{+}_{B}.$$
     On the other hand, since $[\varphi]_{p^{-}}=1$
     $$\inf_{x\in\varphi^{-1}(A\cap B)}p(x)\geq \inf_{x\in\varphi^{-1}(B)}p(x)\geq \inf_{x\in B}p(x)=p^{-}_{B}.$$
     Therefore, for $x\in \varphi^{-1}(A\cap B)$ it follows that $\mu(B)^{\frac{p(x)}{p_{\varphi}(x)}}\leq \mu(B)^{\frac{p^{-}(B)}{p^{+}(B)}}.$ Hence, since $p(\cdot)\in LH_{0}(X)$ by \cite[lemma~3.6]{harjulehto2004variable} there exists $C_{0}>0$ such that $$\mu(B)^{1-\frac{p(x)}{p_{\varphi}(x)}}\geq \mu(B)^{1-\frac{p^{-}_{B}}{p^{+}_{B}}}= \mu(B)^{\frac{p^{+}_{B}-p^{-}_{B}}{p^{+}_{B}}}\geq C_{0}.$$
    \end{proof}

\begin{theorem}\label{t3} 
Let $\mu$ be a Ahlfors $Q$-regular and doubling measure on $X$. 
The non-trivial bounded composition operator $C_{\varphi}$ is not compact on $L^{p(\cdot)}(X).$
\end{theorem}

\begin{proof}
    Note initially that the space $(X,d,\mu)$ does not contain atoms. Indeed, if $E\in \mathcal{B}_{X}$ is an atom then $\mu(E)>0$ so there exists $a\in E$ such that $\mu(E\cap B(a,r))>0$ for all $r>0.$ Hence, if $$\mu(E\cap B(a,r))<\mu(E), \ \  \textrm{for some } \ \ r>0,$$ then the atomicity if $E$ implies that $\mu(E\cap B(a,r))=0,$ it is a contradiction. On the other hand, if $\mu(E\cap B(a,r))=\mu(E)$ for all $r>0$ then by using the $Q$-Ahlfors property of $\mu$ we obtain $0<\mu(E)\leq C (2r)^{Q}$ thus  $\mu(E)=0$ as $r\rightarrow0^{+}$ which is also contradiction.  Besides, suppose that $C_{\varphi}$ is compact on all $L^{p(\cdot)}(X)$ and we consider for $\epsilon\in(0,+\infty)$ the set $$U_{\epsilon}:=\{x\in X: u_{\varphi}(x)>\epsilon\}.$$
    If $\mu(U_{\epsilon})>0$ for some $\epsilon>0$ then from non-atomicity of $\mu$ it follows that there exists a decreasing sequence $\{U_{n}\}$ such that 
      $$\{U_{n}\}\subset U_{\epsilon} \ \ \textrm{with} \ \  0<\mu(U_n)<1/n \ \  \textrm{for any,} \ \ n\in\mathbb{N}.$$ 
      Let's construct a bounded sequence in $ L   ^{p(\cdot)}-$norm that is not equi-integrable in $L^{p(\cdot)};$ given $\delta>0,$ since $\varphi:X\rightarrow X$ is non-singular,  choose $\rho>0$ such that $$\mu(\varphi^{-1}(S))< \delta, \ \ \textrm{whenever} \ \ \mu(S)< \rho.$$
      For this $\rho>0$ there exist $N>0$ which $1/N\leq \rho.$ Now, since $\mu(U_{n})>0$ for all $n\in\mathbb{N}$ by \cite[Lemma~3.3.31]{heinonen2015sobolev} we can fix $x_{n}\in U_{n}$ satisfying $\mu(U_{n}\cap B(x_{n},\delta))>0.$ Consider the set $A_{\delta}\in \mathcal{B}_{0},$ given by $$A_{\delta}:=\varphi^{-1}(U_{N}\cap B(x_{N},\delta)),$$
      it is clear that $\mu(A_{\delta})< \delta.$ Moreover by using Lemma \ref{l2} the function  $f_{N}:X\rightarrow\mathbb{R}$ given by  $$f_{N}:=\mu(B(x_{N},\delta))^{-\frac{1}{p(\cdot)}}\mathcal{X}_{U_{N}\cap B(x_{N},\delta)}(\cdot)\in B_{L^{p(\cdot)}},$$
    is such that
    $$ \displaystyle\int_{A_{\delta}}(C_{\varphi}f_{N})(x)^{p(x)} \ d\mu(x)\geq  \ \int_{\varphi^{-1}(U_{N}\cap B(x_{N},\delta))} \left(\frac{1}{\mu(B(x_{N},\delta))}\right)^{\frac{p(x)}{p_{\varphi}(x)}}\ d\mu(x)$$$$\geq  \mu(B(x_N,\delta))^{\frac{p_{B}^{+}-p_{B}^{-}}{p_{B}^{+}}} \int_{U_{N}\cap B(x_N,\delta)} \ \frac{u_{\varphi}(x)}{\mu(B(x_N,\delta))} \  d\mu(x) $$ $$ \geq\epsilon \ C_{0} \  \frac{\mu(U_{N}\cap B(x_N,\delta))}{\mu(B(x_N,\delta))}. $$
      By differentiation, $$\displaystyle\lim_{r\rightarrow 0^{+}} \frac{\mu(U_{N}\cap B(x_{N},r))}{\mu(B(x_{N},r))}=1.$$
    Therefore, taking $\delta\in (0,+\infty)$ enough small we can suppose that we have the follows lower bound $$\frac{\mu(U_{N}\cap B(x_N,\delta))}{\mu(B(x_N,\delta))}\geq \frac{1}{2}.$$ Thus, $$\displaystyle\int_{A_{\delta}}(C_{\varphi}f_{N})(x)^{p(x)} \ d\mu(x)\geq  \ \epsilon \ C_{0}/2, \ \ \textrm{for} \ \ \delta \ \ \textrm{enough \ small}.$$
    y switching to a subsequence if necessary, we have that the sequence $\{C_{\varphi} f_n\}$ is not equi-integrable in $L^{p(\cdot)}(X)$. Thus, by virtue of \cite[Theorem~1]{gorka2015almost}, this contradicts the compactness of $C_{\varphi}$. Consequently, $\mu(U_{\epsilon})=0$ for all $\epsilon\in(0,+\infty)$ or equivalently $u_{\varphi}(x)=0$ a.e. in $x\in X.$ This implies that, $$\displaystyle\int_{\varphi^{-1}(B)} |C_{\varphi}f|^{p(x)} \ d\mu(x)\leq \sum_{j=1,2}\int_{\varphi^{-1}(B)} |C_{\varphi}f_{j}|^{p^{\pm}} \ d\mu(x)$$$$\leq \sum_{j=1,2}\int_{B}|f_{j}(x)|^{p^{\pm}} \ u_{\varphi}(x) \ d\mu(x)=0,$$
      that is, $C_{\varphi}f=0$ on $\varphi^{-1}(B)$ for all $B\in\mathcal{B}_{0}$ and it is enough for to get $C_{\varphi}=0.$ 
\end{proof}

\section{Some properties for \texorpdfstring{$T_{\varphi}$}{lg}}\label{properties}
\subsection{\texorpdfstring{$L^{p(\cdot)}$}{lg}-boundedness of \texorpdfstring{$T_{\varphi}$}{lg}} 

In this second part, we provide a complete characterization of the continuity for composition operators in the framework of variable Lebesgue spaces.

\begin{theorem}\label{t2}
    Let $\mu$ with doubling property on $X$,      $p(\cdot)\in \mathcal{P}(X)$, and $\varphi:X\rightarrow X$ be a non-singular Borel  measurable map.
    Then, the composition operator $T_{\varphi}$ maps space $L^{p(\cdot)}(X)$ into $L^{p_{\varphi}(\cdot)}(X)$ if and only if the function $x\mapsto u_{\varphi}(x)^{1/p(x)}$ is essentially bounded, that is,  $u_{\varphi}(\cdot)^{1/p(\cdot)}\in L^{\infty}(X).$
    Moreover, $$\left\|T_{\varphi}\right\|=\esssup_{x\in X}\left\{u_{\varphi}(x)^{\frac{1}{p(x)}}\right\}.$$
\end{theorem}

  \begin{remark}
     Theorem \ref{t2} generalizes the well-known result in the framework of Lebesgue spaces with constant exponent. More precisely, when $p(\cdot)=p\geq1$  then both the induced exponent and the space induced by the measurable map $\varphi$ remain invariant, that is, $$p_{\varphi}(\cdot)=p \ \ \textrm{and} \ L^{p_{\varphi}(\cdot)}(X)=L^{p}(X).$$ 
  \end{remark}

\begin{remark}
    Since $f\in L^{p(\cdot)}(X)$ if and only if $f^{p(\cdot)-1}\in L^{p'(\cdot)}(X)$ (the same is true for $p'(\cdot)$), by Theorem \ref{t2}, it follows that $T_{\varphi}$ maps $L^{p(\cdot)}(X)$ into $L^{p_{\varphi}(\cdot)}(X)$ if and only if $T_{\varphi}$ maps $L^{p'(\cdot)}(X)$ into $L^{p'_{\varphi}(\cdot)}(X).$ An interesting question is whether this is true for the operator $C_{\varphi}$ acting on $L^{p'(\cdot)}$ and $L^{p(\cdot)}.$
\end{remark}

\begin{remark}\label{r1}
      In the case that  \begin{equation}\label{c2}
      L^{p_{\varphi}(\cdot)}(X)\hookrightarrow L^{p(\cdot)}(X)
      \end{equation}
      Theorem \ref{t2} provided a sufficient condition which the operator $T_{\varphi}$ maps $L^{p(\cdot)}(X)$ into itself. In fact, the embedding (\ref{c2}) provided a class of maps $\varphi:X\rightarrow X$ induced composition operators on $L^{p(\cdot)}(X).$   
  \end{remark}

\begin{remark}
      Note that from Remark (\ref{r1}) the embedding condition (\ref{c2}) can be modified by assuming that the variable exponent decays to infinity; for example assume $p_{\varphi}(x)\geq p(x)$ a.e. in $x\in X$ (e.g., see {\rm{\cite[Theorem~2.45]{cruz-uribe2013}}}) and $p(\cdot)\in LH_{\infty}(X).$  
      
  \end{remark}

  \begin{remark}
      If the function $x\mapsto u_{\varphi}(x)$ is bounded then for any $p(\cdot)\in\mathcal{P}(X)$ the function  $x\rightarrow u_{\varphi}(x)^{1/p(x)}$ is also bounded so by Theorem \ref{t2} we obtain the following weak inequality: there exists $C>0$ depending of $\varphi$ such that for each $t>0$ and $f\in L^{p(\cdot)}(X),$
    $$\left\|t \ \mathcal{X}_{\{x:(f\circ\varphi)(x)>t\}}\right\|_{p_{\varphi}(\cdot)}\leq C \  \|f\circ\varphi\|_{p_{\varphi}(\cdot)}\leq C \ \|f\|_{p(\cdot)}$$
  \end{remark}

\begin{proof} (Proof of Theorem \ref{t2}) Suppose that $T_{\varphi}$ maps $L^{p(\cdot)}(X)$ into $L^{p_{\varphi}(\cdot)}(X)$ then by \textit{Closed Graph Theorem} there exists $C>0$ such that $\|T_{\varphi}\|\leq C.$ Let us show that  $u_{\varphi}(\cdot)^{1/p(\cdot)}\in L^{\infty}(X);$ note first that $u_{\varphi}(\cdot)^{1/p(\cdot)}\in L^{1}_{loc}(X)$ so let $B\in \mathcal{B}_{0}$ and define the function 
$f:X\rightarrow \mathbb{R}$ by 
\begin{equation}\label{n1}
f(x)=\displaystyle\mu(B)^{-\frac{1}{p(x)}} \ \mathcal{X}_{B}(x), \ \ x\in\mathbb{R}.
\end{equation}
Hence, it is clear $\rho(f/\lambda)\leq 1$ for all $\lambda\geq 1$ so $f\in L^{p(\cdot)}(X)$ and $\|f\|_{p(\cdot)}=1$ this implies $f\circ\varphi\in L^{p_{\varphi}(\cdot)}(X)$ and $\|f\circ\varphi\|_{p_{\varphi}(\cdot)}\leq C,$ that is, there exists $\lambda_{0}>0$ which $\lambda_{0}<C$ and $$ \displaystyle\int_{\varphi^{-1}(B)} \lambda_{0}^{-p_{\varphi}(x)} \ \mu(B)^{-\frac{p_{\varphi}(x)}{p_{\varphi}(x)}} \ dx\leq 1.$$  
Consequently, 
\begin{eqnarray*}
    \frac{1}{\mu(B)}\displaystyle\int_{B} C^{-p(x)} \ u_{\varphi}(x) \ dx &=& \frac{1}{\mu(B)}\displaystyle\int_{\varphi^{-1}(B)} C^{-p_{\varphi}(x)}  \ dx \\ 
    &\leq&  \displaystyle\int_{\varphi^{-1}(B)} \lambda_{0}^{-p_{\varphi}(x)} \ \mu(B)^{-\frac{p_{\varphi}(x)}{p_{\varphi}(x)}} \ dx \\
    &\leq& 1,
 \end{eqnarray*}
 by differentiation, it follows that the map $x\mapsto C^{-p(x)} u_{\varphi}(x)$ is essentially bounded. Reciprocally,  denote by $M_{\varphi}$ the multiplication operator with symbol $u_{\varphi}(\cdot)^{1/p(\cdot)}$ so if $u_{\varphi}(\cdot)^{1/p(\cdot)}\in L^{\infty}(X)$ then for each $f\in L^{p(\cdot)}(X),$ $$\left|(M_{u_{\varphi}}f)(x)\right|\leq  \esssup_{y} \left\{u_{\varphi}(y)^{1/p(y)}\right\} \ |f(x)|, \ \ \textrm{a.e. in} \ x\in X.$$
 Since $L^{p(\cdot)}(X)$ is a lattice, 
 $$\|M_{\varphi}f\|_{p(\cdot)}\leq \esssup_{y} \left\{u_{\varphi}(y)^{1/p(y)}\right\} \|f\|_{p(\cdot)}.$$
 Hence, 
 \begin{eqnarray*}
     \|f\circ\varphi\|_{p_{\varphi}(\cdot)}&=&\inf\left\{\lambda>0: \displaystyle\int_{X} \lambda^{-p(\varphi(x))} \ |f(\varphi(x))|^{p(\varphi(x))} \ dx \leq 1\right\} \\
     &=&\inf\left\{\lambda>0: \displaystyle\int_{X} \lambda^{-p(x)} \ |f(x)|^{p(x)} \ u_{\varphi}(x) \ dx \leq 1\right\}\\
     &=& \inf\left\{\lambda>0: \displaystyle\int_{X} \lambda^{-p(x)} \ |(M_{u_{\varphi}}f)(x)|^{p(x)} \ dx \leq 1\right\} \\
&=&\|M_{u_{\varphi}}f\|_{p(\cdot)} \\
&\leq& \esssup_{y} \left\{u_{\varphi}(y)^{1/p(y)}\right\} \|f\|_{p(\cdot)}
 \end{eqnarray*}
 that is, $T_{\varphi}$ maps $L^{p(\cdot)}(X)$ into $L^{p_{\varphi}(\cdot)}(X)$ and $\|T\|\leq  \esssup_{y} \left\{u_{\varphi}(y)^{1/p(y)}\right\},$ computing with the normalized functions given in the equality of the norm (\ref{n1}). 
  \end{proof}

\subsection{Compactness for \texorpdfstring{$T_{\varphi}$}{lg}} In the study of the compactness of $C_{\varphi}$ over $L^{p(\cdot)}$ the presence of non-atomic sets was a consequence of the regular Ahlfors structure of the space necessary to have at least one continuous composition operator. For the case of the operator $T_{\varphi}$, to obtain continuity we only require that the space admits a doubling measure which is a weaker hypothesis than the Ahlfors regularity therefore to guarantee no atomic sets we assume in addition that the space is connected. 

\begin{lemma}
\label{l3} Every connected doubling metric measure space $(X,d,\mu)$ does not contain atoms.   \end{lemma}
\begin{proof}
    Assume that  $E\in \mathcal{B}_{0}$ is an atom then $\mu(E)>0$ this implies $\mu(E\cap B(w,r))>0$ for some $w\in E$ and all $r>0.$ Hence, in the case $\mu(E\cap B(w,r))=\mu(E)$ by using \cite[Lemma~ 3.7]{bjorn2011nonlinear} there exist constants $C,\sigma>0$ such that $$\mu(E)\leq \mu(B(w,r))\leq C \ r^{\sigma} \ R^{-\sigma} \ \mu(B(w,R)), \ R>r>0,$$
    from here that, $\mu(E)=0$ as $r\rightarrow0^{+},$ and it is a contradiction. On other case, $\mu(E\cap B(w,r))<\mu(E)$ but the atomicity of $E$ implies that $\mu(E\cap B(w,r))=0,$ contradiction. 
\end{proof}
 As a consequence of Lemma \ref{l3}, the following result is well-know in the framework the non-atomic metric measure spaces (see, \cite[Theorem~4.2]{bajaj2024composition} and \cite[Theorem~ 5.2 and Theorem~5.3]{datt2024weighted}). 
 
\begin{theorem}\label{t5} Assume that $(X,d,\mu)$ is a metric measure space with doubling measure, $r(\cdot)\in \mathcal{P}(X)$ such that $1\leq p^{-}\leq p^{+}<+\infty,$ and $\varphi:X\rightarrow X$ Borel non-singular map. If $X$ is connected then the space $L^{r(\cdot)}(X)$ does not admit compact composition operators $T_{\varphi}.$
\end{theorem}
\subsection{Weak compactness of $T_{\varphi}$ on $L^{p(\cdot)}([0,1])$}
     
In the case $1<p^{-}\leq p^{+}<+\infty$ is well-know that $L^{p(\cdot)}$ is reflexive space, so every composition operator on $L^{p(\cdot)}$ is weakly compact. The non-reflective case $(p^{-}=1)$ is different and has been explored in \cite{Takagi1992} for the constant exponent case. In the following theorem we provide some results in this direction.
       
Denote by $\lambda$ the Lebesgue measure on $[0,1]$ or $\mathbb{R}$ and $\Omega_{1}:=p^{-1}(\{1\})$ for an exponent $p(\cdot)$.
       
   \begin{theorem}\label{t4} 
    Let \(r(\cdot) \in \mathcal{P}([0,1])\) with \(\lambda(\Omega_{1}) = 0\) 
    and \(r^{+} < \infty\). Let \(\varphi: [0,1] \rightarrow [0,1]\) 
    be a non-singular map such that 
    \[ 
    T_{\varphi}: L^{r(\cdot)}([0,1]) \rightarrow L^{r_{\varphi}(\cdot)}([0,1]). 
    \] 
    Then, the following properties hold:
    
    \begin{enumerate}[label=\textnormal{(w.\arabic*)}]
        \item \label{w1} \(T_{\varphi}\) maps relatively weakly compact subsets into relatively weakly compact subsets.  
        
        \item \label{w2} Let \(v_{\varphi}(\cdot) := r(\cdot)^{1/p(\cdot)}\). The operator \(T_{\varphi}\) is relatively weakly compact if and only if the multiplication operator \(M_{v_{\varphi}}\) is relatively weakly compact.

        \item \label{w3} Let \(r^{-} = 1\). If 
        \[ 
        M := \inf_{z \in \Omega \setminus \Omega_1} u_{\varphi}(z) > 0, 
        \] 
        then the operator \(T_{\varphi}: L^{r(\cdot)}([0,1]) \rightarrow L^{r_{\varphi}(\cdot)}([0,1])\) is not weakly compact.
    \end{enumerate}
\end{theorem}       
        
\begin{proof} To prove $(w1)$, assume that $M>0$ then since $r^{-}=1$ we may choose a sequence $(z_n)\subset[0,1]\setminus\Omega_{1}$  satisfying $r(z_n)\rightarrow1$ when $n\rightarrow+\infty.$ So, denote by $$B:=B_{L^{r(\cdot)}}=\{f\in L^{r(\cdot)}([0,1]):\|f\|_{r(\cdot)}\leq 1\}.$$ Let us reason by contradiction, if $T_{\varphi}$ is weakly compact then the subset $T_{\varphi}(B)$ is relatively weakly compact in $L^{r_{\varphi}(\cdot)}([0,1]).$ Hence, by \cite[Theorem~4.3]{HernandezRuizSanchiz21} we have 
\begin{equation}\label{wa1}
\displaystyle\lim_{\lambda\rightarrow 0^{+}}\sup_{f\in B}\lambda^{-1}\int_{[0,1]} |\lambda|^{r_{\varphi}(z)} \ |(T_{\varphi}f)(z)|^{r_{\varphi}(z)} \ dz=0,
\end{equation}
Since, 
\begin{equation}\label{wa2}
\lambda^{-1}\int_{[0,1]} |\lambda|^{r_{\varphi}(z)} \ |(T_{\varphi}f)(z)|^{r_{\varphi}(z)} \ dz=\lambda^{-1}\int_{[0,1]} |\lambda|^{r(z)} \ |f(z)|^{r(z)} \ u_{\varphi}(z) \ dz, 
\end{equation}
by (\ref{wa1}) we get 
\begin{equation}\label{wa3}
\displaystyle\lim_{\lambda\rightarrow 0^{+}}\sup_{f\in B}\lambda^{-1}\int_{[0,1]} |\lambda|^{r(z)} \ |f(z)|^{r(z)} \ u_{\varphi}(z) \ dz=0.
\end{equation}
Now, let $\mathcal{I}_{0}':=\{[a,b]\in\mathcal{I}_0:[a,b]\subset[0,1]\}$ and  we consider the functions $f:\Omega\rightarrow\mathbb{R}$ given by  $$f(x):=f_{ab}(x):=\displaystyle(b-a)^{-1/r(x)} \  \mathcal{X}_{[a,b]}(x), \ \ x\in[0,1].$$
    So, it is clear that $\{f_{ab}:[a,b]\in\mathcal{I'}_{0}\}\subset B$ and thus from (\ref{wa3}) given $\epsilon>0$ there exists $\lambda_0>0$ small such that $$\frac{1}{b-a}\int_{a}^{b} \lambda_{0}^{r(z)-1} \  \ u_{\varphi}(z) \ dz<\epsilon, \ \ \forall \ [a,b]\in\mathcal{I'}_{0}$$
    by differentiation, $$\lambda_{0}^{r(z)-1} \ u_{\varphi}(z)<\epsilon, \ \ \textrm{a.e. in } \ z\in[0,1]\setminus\Omega_{1}.$$
    In particular,  taking $z=z_n$ and $n\rightarrow+\infty$ we obtained $M<\liminf u_{\varphi}(z_n)=0$ because $\epsilon$ is taken arbitrarily so $M<0$ which is a contradiction. The property $(w2)$ follows from (\ref{wa2}), \cite[Theorem~4.3]{HernandezRuizSanchiz21} and of the fact $\mathcal{U}(\varphi)<\infty.$ Finally, the property $w3$ follows of \ref{wa2} and \cite[Theorem~4.3]{HernandezRuizSanchiz21}. \end{proof} 
    
     \begin{remark}
     The proof of Theorem {\rm{\ref{t4}}} can be easily extended to $L^{r(\cdot)}(\mathbb{R})$ by applying results recently obtained in {\rm{\cite[section~5]{HernandezRuizSanchiz22}} }under the restriction that  $\Omega_1$ be a null-set. 
    \end{remark}
     \textbf{\textit{The case $\lambda(\Omega_1)>0$ and $r^{-}=1$}:} Suppose that $\varphi^{-1}(\Omega_1)\subset\Omega_1,$ let us choose $z_0\in\mathbb{R}$ such that $p(z_0)=1$ and for each $n\in\mathbb{N}$ define $I_n:=(z_0-1/2^{n},z_0+1/2^{n})$ so $I_n\cap\Omega_1\neq \emptyset$ for each $n\in\mathbb{N}.$ We show that $T_{\varphi}(B)$ is not relatively weakly compact provided that $M>0$; define the sequence of measurable functions $\{f_n\}$ by 
    \begin{center}
        $f_n(x):=      \displaystyle\frac{1}{\lambda(I_n)} \ \mathcal{X}_{I_n\cap\Omega_1}(x), \ \ x\in\mathbb{R}.$    
    \end{center}
    
It is clear that $f_n\in B$ for every $n\in\mathbb{N}.$ However, from $u_{\varphi}(x)\geq M$ a.e. $x\in\mathbb{R}$ we obtain $$\displaystyle\int_{\varphi^{-1}(I_n)\cap\Omega_1} (f_n \circ\varphi)(x) \ dx\geq  \int_{I_n\cap\Omega_1} f_{n}(x) \ u_{\varphi}(x) \ dx\geq M \ \frac{\lambda(I_n\cap\Omega_1)}{\lambda(I_n)}.$$
   Taking a subsequence if necessary, let us make  $\lambda(I_n)\rightarrow0$ this means $\lambda(\varphi^{-1}(I_n))\rightarrow0$ and by differentiation $$\displaystyle\frac{\lambda(I_n\cap\Omega_1)}{\lambda(I_n)}\rightarrow1, \ \ \ n\rightarrow+\infty$$
   Therefore, $$\liminf \displaystyle\int_{\varphi^{-1}(I_n)\cap\Omega_1} (f_n \circ\varphi)(x) \ dx\geq M>0, \ \ \ \mu(\varphi^{-1}(I_n))\rightarrow0.$$
   Applying \cite[Proposition~5.11] {HernandezRuizSanchiz22} we yields the assertion for $T_{\varphi}(B).$ 



\section*{Acknowledgments}
The authors thank the anonymous referees for the suggestions to improve this article. Carlos F. Álvarez was partially supported by inner project BASEX-PD/2023-02 of University of Sinú, Cartagena. 


\bibliography{Bibliography}

\bibliographystyle{acm}

\end{document}